\documentclass[12pt]{article}

\usepackage[margin=1in]{geometry}  
\usepackage{graphicx}              
\usepackage{amsfonts}              
\usepackage{amssymb}
\usepackage{amsthm}                
\usepackage{amsmath}               
\usepackage{color}
\usepackage{tabulary}
\usepackage{float}
\usepackage{diagrams}

\newtheorem{theorem}{Theorem}[section]
\newtheorem{lemma}[theorem]{Lemma}
\newtheorem{proposition}[theorem]{Proposition}
\newtheorem{corollary}[theorem]{Corollary}
\newtheorem{example}[theorem]{Example}

\newenvironment{remark}[1][Remark]{\begin{trivlist}
\item[\hskip \labelsep {\bfseries #1}]}{\end{trivlist}}

\DeclareMathOperator{\Atp}{Atp}
\DeclareMathOperator{\Aut}{Aut}

\newcommand{\ZZ}{\mathbb{Z}}      

\title{About the autotopisms of abelian groups}
\author{Lucien Clavier}

\begin{document}
\maketitle
\tableofcontents

\pagebreak
\section{Introduction}
\label{introduction}

An \textit{isotopism} between two groups $(G,\cdot)$ and $(H,\ast)$, or more generally between two quasigroups (see \cite{Pflu}) is a triple $t=(\alpha, \beta, \gamma)$ of bijections from $G$ to $H$ verifying for each $x,y,z \in G$:
\[
\alpha(x)\ast \beta(y)=\gamma(x\cdot y).
\]
If such a triple exists, $G$ and $H$ are said to be \textit{isotopic}.

We call \textit{isomorphism} an isotopism $t=(\alpha, \beta, \gamma)$ where $\alpha=\beta=\gamma$.

Suppose $G$ and $H$ are finite, and defined on the same set, say $\{1\ldots n\}$. Then we call \textit{principal isotopism} an isotopism $t=(\alpha, \beta, \gamma)$ where $\gamma$ is the identity on $\{1\ldots n\}$. 

As usual, if $G=H$, we say respectively \textit{autotopism, automorphism} instead of isotopism, isomorphism.
We write \textit{$\Atp(G)$} for the set of all autotopisms of $G$ and \textit{$\Aut(G)$} for the set of all automorphisms of $G$. These are groups with respect to the pointwise law of composition 
\[
(\alpha, \beta, \gamma)\cdot (\alpha', \beta', \gamma')=(\alpha \circ \alpha', \beta\circ\beta', \gamma\circ\gamma')
\]
and we note $1$ their neutral element. Of course, we write $h$ instead of $(h,h,h)$ for an element of $\Aut(G)$. 

The reason why the study of isotopisms has not been developed in Group theory is that if any two groups are isotopic, then they must be isomorphic (see \cite{Pflu}). However, it remains a powerful tool in Quasigroup and Loop theory. From \cite{07}, we know explicitly $\Aut(G)$ when $G$ is a finite abelian group. We would like to link $\Atp(G)$ and $\Aut(G)$ in that case. 

Here is a summary of the paper, with $G$ a finite abelian group throughout:

2. We can always decompose any autotopism of $G$ into an automorphism and a \textit{normalized} autotopism, i.e. an autotopism of the form 
\[
\left\{
    \begin{array}{lll}
		\alpha_0 :x &\mapsto x+x_0 \\
		\beta_0 :y &\mapsto y+y_0\\
		\gamma_0 :z & \mapsto z+x_0+y_0
    \end{array}
\right.
\]
for some $x_0, y_0 \in G$.
As a consequence, $\Atp(G)$ is the set $ \Aut(G) \times G^2$ with a nice multiplication making it a semi-direct product.

We extend that result to some other algebras, including groups and $CC$-loops.

3. We show Proposition \ref{p3}, linking $\Atp(G)$, $\Aut(G)$ and $G^2$. 
We investigate some chosen cases: when $\Aut(G)$ is cyclic, and when $G=\mathbb{Z}_{2^r}$, $r>2$. In the special case when $G=\mathbb{Z}_{p}$ with $p$ a prime integer, we show that, up to conjugacy, any subgroup of $\Atp(G)$ is exactly the product of a subgroup of $\Aut(G)$ and a subgroup of $G^2$.

4. We reduce the problem to the one where $G$ is a $p$-group for some prime integer $p$. As a result, we know in detail the subgroup structure of $\Atp(\mathbb{Z}_n)$ for any integer $n$; this is Theorem \ref{zn}.

\pagebreak
\section{$\Atp(G)$ as a semidirect product}
\label{semidir}
\subsection{Abelian groups}

In this section, we prove Proposition \ref{proposition}, which is the key ingredient in all the following sections.
First, for any finite abelian group $G$, for all $h \in \Aut(G)$ and all $X=(x_0,y_0)\in G^2 $, define an autotopism $t_{h,x_0,y_0}$ by
\[
\begin{cases}
	x \mapsto hx+x_0 \\
	y \mapsto hy+y_0\\
	z \mapsto hz+x_0+y_0
\end{cases} 
\]

\begin{proposition}
\label{proposition}
Let $G$ be a finite abelian group. Then
\begin{align*}
\phi: \Aut(G) \ltimes G^2 &\rightarrow \Atp(G) \\
(h,x_0,y_0) &\mapsto t_{h,x_0,y_0}
\end{align*}
is an isomorphism, where the multiplication on $ \Aut(G) \ltimes G^2$ is given by
\[
(h,X)(h',X')=(hh',hX'+X)
\]
\end{proposition}

\begin{proof}
It is straightforward to check that $\phi$ is a group homomorphism, and that it is injective. Let us prove surjectivity.

Let $t=(\alpha, \beta, \gamma)$ be an autotopism of $G$. Define the loop $(L,\ast)$ on $\{1, \ldots, |G| \}$ so that $\gamma$ is an isomorphism between $G$ and $L$. Then
\[
\overline{t}=(\overline{\alpha},\overline{\beta},\overline{\gamma})=(\alpha\circ\gamma^{-1},\beta \circ\gamma^{-1},\text{Id}) 
\]
is a principal isotopism from $L$ to $G$. This is to say that for each $x,y \in L$
\[
\overline{\alpha}(x)+\overline{\beta}(y)=x\ast y
\]
We are interested in understanding the multiplication on $L$.
But, setting $x_0=\overline{\alpha}(e)$ and $y_0=\overline{\beta}(e)$, where $e$ is the neutral element of $L$, we get
\[
\begin{cases}
\overline{\alpha}(x)+y_0=x\ast e=x \\
x_0+\overline{\beta}(y)=e\ast y =y
\end{cases}
\]
Thus $\ast$ is simply defined by
\[
x\ast y=x+y-x_0-y_0
\]
Therefore, $G$ is also isomorphic to $L$ via $h_0 : x \mapsto x+x_0+y_0$. 
Define $h=h_0^{-1} \gamma \in \Aut(G)$.

\begin{diagram}
G &\rTo^{h_0} &L\\
\uTo^{h} &\ruTo^{\gamma} &\dTo_{\overline{t}}\\
G &\rTo^{t} &G
\end{diagram}

Now, by composition, we can write $t$ as
\[
\left\{
    \begin{array}{lll}
		\alpha :x &\mapsto hx+x_0 \\
		\beta :y &\mapsto hy+y_0\\
		\gamma :z & \mapsto hz+x_0+y_0
    \end{array}
\right.
\]
and we are done.

\end{proof}

\subsection{Other varieties}
Before investigating in more details the case of abelian groups, let us generalize this construction to broader varieties. 
This subsection is not needed for the rest of the article, and is mostly designed for quasigroup theorists.

Adopting the universal-algebraic viewpoint, recall that a \textit{quasigroup} is an algebra 
\[
<Q,\cdot, \backslash, /> 
\]
satisfying the two following pairs of identities:
\[
\begin{cases}
Q_1:\; x\backslash (x \cdot y) = y;\; (x  \cdot y)/y = x \\
Q_2:\; x \cdot (x\backslash y) = y;\; (x/y) \cdot y = x 
\end{cases}
\]

A \textit{loop} is a quasigroup with a neutral element, i.e. an algebra $<Q,\cdot, \backslash, /,1>$ which satisfies $Q_1$, $Q_2$ and 
\[
Q_3:\; x \cdot 1 = 1 \cdot x = x
\]

Instinctively, loops are ``non-necessary associative groups''.
In quasigroup theory, we are often interested in subvarieties of loops ``close to being groups'' (for instance Bol-Moufang loops, see \cite{Petr2}).

Let $\eta=\eta(x,x_0,y_0)$ be an expression in which $x$ appears exactly once (e.g. $\eta(x,x_0,y_0)=x_0\cdot xy_0$ or $\eta(x,x_0,y_0)=x_0y_0\cdot(xy_0)/x_0$).

Define $i(\eta)$ to be the identity 
\[
i(\eta):\; \eta(x,x_0,y_0)/y_0 \cdot x_0\backslash \eta(y,x_0,y_0)=\eta(xy,x_0,y_0)
\]

Then we can generalize Proposition \ref{proposition} to the variety $V(\eta)$ defined by $Q_1$, $Q_2$, $Q_3$ and $i(\eta)$.

\begin{proposition}
\label{p2}
Let $G$ be a loop in $V(\eta)$. Then the application $\Phi$ from $\Aut(G)\times G^2$ to $\Atp(G)$ defined by
\[ 
\Phi(h,x_0,y_0) = \begin{cases}
	x \mapsto \eta(h(x),x_0,y_0)/y_0 \\
	y \mapsto x_0\backslash \eta(h(y),x_0,y_0) \\
	z \mapsto \eta(h(z),x_0,y_0)
\end{cases}
\]
is a bijection.

Equipping $\Aut(G)\times G^2$ with the multiplication
\[
 (h_0,x_0,y_0)(h_1,x_1,y_1)=(h_2,x_2,y_2)
\]
where
\[
 \begin{cases}
  \varphi_0:x\mapsto\eta(h_0(x),x_0,y_0)\\
  \varphi_1:x\mapsto\eta(h_1(x),x_1,y_1)\\
  x_2=\varphi_0(x_1)/y_0\\
  y_2=x_0\backslash \varphi_0(y_1)\\
  \varphi_2:x\mapsto\eta(x,x_2,y_2)\\
  h_2=\varphi_2^{-1}\circ \varphi_0 \circ \varphi_1
 \end{cases}
\]
$\Phi$ becomes an isomorphism.

Moreover, $\Phi^{-1}$ is simply given by 
\[
 \Phi^{-1}(\alpha,\beta,\gamma)=(h_0^{-1}\circ\gamma,\alpha(1),\beta(1))
\]
where $h_0$ is defined by
\[
 h_0(x)=\eta(x,\alpha(1),\beta(1))
\]
\end{proposition}
\begin{proof}
We mimic the proof of Proposition \ref{proposition}. 

Let $t=(\alpha, \beta, \gamma)$ be an autotopism of $G$.

Define the loop $(L,\ast)$ on $\{1, \ldots, |G| \}$ so that $\gamma$ is an isomorphism between $G$ and $L$.

Define 
\[
\overline{t}=(\overline{\alpha},\overline{\beta},\overline{\gamma})=(\alpha\circ\gamma^{-1},\beta \circ\gamma^{-1},\text{Id}) 
\]

Now, for each $x,y \in L$
\[
\overline{\alpha}(x) \cdot \overline{\beta}(y)=x\ast y
\]
so, setting $x_0=\overline{\alpha}(e)$ and $y_0=\overline{\beta}(e)$, where $e$ is the neutral element of $L$, we get
\[
\begin{cases}
\overline{\alpha}(x)=x/y_0 \\
\overline{\beta}(y)=x_0\backslash y 
\end{cases}
\]
and finally
\[
x\ast y=x/y_0\cdot x_0\backslash y
\]

We say that $L$ is the \textit{$(x_0,y_0)$-isotope} of $G$. 

From the hypothesis, $G$ is isomorphic to $L$ via 
\[
h_0 : x \mapsto \eta(x,x_0,y_0) 
\]
which is a bijection since $x$ appears exactly once in $\eta$. Define $h=h_0^{-1} \circ \gamma \in \Aut(G)$.
By composition, we can write $t=\overline{t}\circ h_0\circ h$ as
\[ 
\Phi(h,x_0,y_0) = \begin{cases}
	x \mapsto h_0\circ h(x)/y_0 \\
	y \mapsto x_0\backslash h_0\circ h(y) \\
	z \mapsto h_0\circ h(z)
\end{cases}
\]
and we are done for the first part.

The rest of the proof is straightforward, noticing that since $h_0$ is an isomorphism, 
\[
h_0(1)=e=x_0y_0
\]
i.e. $V(\eta)$ satisfies the identity
\[
 \eta(1,u,v)=uv
\]
\end{proof}

The following Corollary is a direct consequence of the previous discussion.
\begin{corollary}
\label{corol}
For every loop $G$ in a variety $V(\eta)$ with $\eta$ as in Proposition \ref{p2},
\[ |\Atp(G)|= |G|^2 |\Aut(G)|\]
\end{corollary}

\begin{remark}
At this point, we could raise the question of which varieties of loops can be written $V(\eta)$ for some $\eta$ as above.
In fact, any $V(\eta)$ must be a variety consisting of $G$-loops only, since every loop in $V(\eta)$ is isomorphic to each of its $(x_0,y_0)$-isotopes (see \cite{Pflu}). 

For the particular choice $\eta=x_0x\cdot y_0$, we get 
\[
i(\eta): x_0x\cdot x_0\backslash (x_0y \cdot y_0)=(x_0\cdot xy)y_0
\]
and taking $x_0=1$ yields
\[
i'(\eta): x\cdot yy_0=xy\cdot y_0
\]
thus $V(\eta)$ reduces to the variety of groups.
Then the multiplication is simply
\[
(h,x_0,y_0)(h',x_1,y_1)=(hh', x_0h(x_1),x_0^{-1}h(y_1)x_0y_0)
\]
and if $G$ is abelian, we find of course the semiproduct law of Proposition \ref{proposition}.

Moreover, we can find an $\eta$ such that $V(\eta)$ is the variety $CCL$ of \textit{CC-loops}, which is the most natural variety consisting of $G$-loops only and bigger than the variety of groups (see \cite{CCL}). Indeed, from the proof of Lemma 2.12 and the discussion leading to Lemma 2.13 in \cite{CCL}, we see that any $CC$-loop verifies $i(\eta_{CCL})$ for $\eta_{CCL}=y_0(x\cdot y_0\backslash x_0 y_0 )$; but conversely, plugging $x_0=1$ or $y_0=1$ in $i(\eta_{CCL})$ shows that any loop in $V(\eta_{CCL})$ is a $CC$-loop, thus
\[
V(\eta_{CCL})=CCL
\]

Details are left to the reader.
\end{remark}

\pagebreak
\section{Subgroup structure of $\Atp(G)$}
\label{struct}

Proposition \ref{proposition} states that $\text{Atp}(G)$ is isomorphic to $\text{Aut}(G) \ltimes G^2$; we are now interested in its subgroup structure.

\begin{proposition}
\label{p3}
Every subgroup of $\text{Aut}(G)\ltimes G^2$ is of the form
\[
({1}\ltimes Q)\cdot <(h_1,X_1),\ldots ,(h_r,X_r)>
\]
where $Q$ is a subgroup of $G^2$, $X_1\ldots X_r\in G^2$ and $\{h_1, \ldots, h_r \}$ is a (minimal) set of generators for some subgroup $H$ of $\text{Aut}(G)$.
Conversely, these are always subgroups of $\text{Aut}(G)\ltimes G^2$.
\end{proposition}
\begin{proof}
Let $S$ be a subgroup of $\text{Aut}(G)\ltimes G^2$.

Define 
\[
H(S)=\{h; \; (h,X)\in S \text{ for some } X\in G^2\}
\]

$H(S)$ is clearly a subgroup of $\text{Aut}(G)$. For each $h$ in $H(S)$, define
\[
Q_h(S)=\{X; \; (h,X)\in S\}
\]

Since 
\[(1,X)(1,Y)=(1,X+Y)\]
for all $X$, $Y$ in $G^2$, we notice that $Q_1(S)$ is a subgroup of $G^2$, making $\{1\}\ltimes Q_1(S)$ a normal subgroup of $S$.

Now, fix some $h\in H(S)$; we see that for all $X, Y, Z$ in $G^2$ such that $(h,X)$, $(h,Y),$ $(1,Z)\in S$, we have 
\[(1,Z)(h,X)=(h,X+Z)\in S\]
and 
\[(h,X)(h,Y)^{-1} = (h,X)(h^{-1},-h^{-1}Y) = (1,X-Y)\in S \]

Thus, $Q_h(S)$ is always a coset of $Q_1(S)$. Choosing for each $h\in H$ an element $X_h\in Q_h$, we can describe $S$ as
\[
S=\{ (h, X_h+ X);\; h\in H(S),\; X\in Q_1(S) \}
\]
i.e.
\[
 S=({1}\ltimes Q_1(S))\cdot <(h,X_h);\;h\in H(S) > 
\]
but since the $X_h$ can be chosen arbitrarily in $Q_h(S)$, this is also
\[
S={(1}\ltimes Q_1(S))\cdot <(h_1,X_{h_1}),\ldots ,(h_r,X_{h_r})>
\]
where $\{h_1 \ldots h_r\}$ is a (minimal) set of generators for $H(S)$.

Conversely, 
\[
({1}\ltimes Q)\cdot <(h_1,X_1),\ldots ,(h_r,X_r)>
\]
is always a subgroup of $\text{Aut}(G)\ltimes G^2$ and we are done.
\end{proof}

We now investigate in more details some simple cases, starting with the case when $\text{Aut}(G)$ is cyclic. We are motivated by the following lemma:

\begin{lemma}
\label{lang}
Let $p$ be a prime integer and $r>0$. Then $A=\mathbb{Z}_{p^r}^\ast$ is cyclic, except in the case when $p=2$ and $r \geq 3$. In the latter case, $A$ is isomorphic to $\mathbb{Z}_{2}\times \mathbb{Z}_{2^{r-2}}$, where $\mathbb{Z}_{2}$ represents the orbit of $-1$ and $\mathbb{Z}_{2^{r-2}}$ the orbit of $5$ in $A=\mathbb{Z}_{2^r}^\ast$. 
\end{lemma}
\begin{proof}
This classical result is supposed well-know. See \cite{Lang}, Ex. 7 p. 79.
\end{proof}

In the following, it is convenient to define
\[
s_k(h)=\left\{
    \begin{array}{ll}
    	1+h+\ldots+h^{k-1} \text{ if } h\neq 1\\
    	0 \text{ otherwise}
    \end{array}
\right.
\]
for all integer $k$. 

\begin{example}
\label{neq2}
Let $G$ be an abelian group such that $\text{Aut}(G)$ is cyclic. Then the general form of a subgroup of $\text{Atp}(G)\ltimes G^2$ is
\[
S= \{ (h^k,s_{k}(h)X+Y);\; k \in \{1 \ldots |h|\},\, Y \in Q \}
\]
where $h\in \text{Aut}(G)$, $X\in G^2$ and $Q$ is some subgroup of $G^2$ containing $s_{|h|}(h)X$.
\end{example}
\begin{proof}
Every subgroup of a cyclic group is cyclic. Moreover,

\begin{align*}
  <(h,X)>&=\{(h^k, s_k(h) X);\; k\in \ZZ \} \\
&=\{(h^k, (s_k(h) + m s_{|h|}(h))X);\;  k \in \{1 \ldots |h|\},\, m \in \ZZ \}
\end{align*}

The proof follows by Proposition \ref{p3}
\end{proof}

The following example gives the subgroups of $\ZZ_p$, up to conjugacy.

\begin{example}
\label{zp}
Let $p$ be a prime integer. Then, up to conjugacy, any subgroup of $\text{Atp}(\ZZ_p)$ is the product of a subgroup of $\text{Aut}(\ZZ_p)=\ZZ_p^\ast$ by a subgroup of $\ZZ_p^2$.

More precisely, Table 1 below provides a set of representatives $S$ for conjugacy classes of $\text{Atp}(\ZZ_p)$, together with their normalizer 
\[
N(S)=\{(n,Y);\;(n,Y)S(n,Y)^{-1}=S\}.
\]
\end{example}

Since $\text{Aut}(\ZZ_p)=<m_0>\cong \ZZ_{p-1}$ is cyclic generated by some automorphism $m_0$, we define $H_d$ for each $d$ dividing $p-1$ to be the (unique) subgroup $H_d=<m_0^d>$ of order $(p-1)/d$.

\begin{table}[H]
\begin{center}
\begin{tabular}{|l|l|}
\hline
representative $S$ & $N(S)$ \\ \hline
$\{1\}=1\ltimes {(0,0)}$  & $\text{Atp}(\mathbb{Z}_p)$\\ 
$1\ltimes <(1,y)>$  &    \\ 
$1\ltimes <(0,1)>$  &    \\ \hline
$H_d,\, d\neq p-1$ & $\text{Aut}(\mathbb{Z}_p)$ \\ \hline
$H_d \ltimes<(1,y)>, \, d\neq p-1$ & $\text{Aut}(\mathbb{Z}_p)\ltimes <(1-m_0^d)^{-1}\cdot(1,y) >$  \\  \hline
$H_d \ltimes<(0,1)>, \, d\neq p-1$ & $\text{Aut}(\mathbb{Z}_p)\ltimes <(1-m_0^d)^{-1}\cdot(0,1) > $   \\ \hline
$H_d\ltimes \mathbb{Z}_p^2$  & $\text{Atp}(\mathbb{Z}_p)$   \\ \hline
\end{tabular}
\end{center}
\caption{Representatives for conjugacy classes of $\text{Atp}(\mathbb{Z}_p)$ and their normalizer.}
\end{table}
\begin{proof}
Since
\begin{align*}
(n,Y)(m,X)(n,Y)^{-1}&=(n,Y)(m,X)(n^{-1},-n^{-1}Y)  \\
					&=(m,nX+(1-m)Y)
\end{align*}
the proof is straightforward; it is left to the reader.
\end{proof}

Let us now describe the case $G= \mathbb{Z}_{2^r}$, $r>2$. 
\begin{example}
\label{eq2}
Let $G= \mathbb{Z}_{2^r}$, $r>2$. Let $S$ be a subgroup of $\text{Atp}(G)$. Then either
\[
S= \{ (h^k,s_{k}(h)X+Z);\; k \in \{1 \ldots |h|\},\, Z \in Q \}
\]
where $h\in \text{Aut}(G)$, $X\in G^2$ and $Q$ is some subgroup of $G^2$ containing $s_{|h|}(h)X$, or
\[
S=\{ (-1)^\epsilon h^k,s_k(h)X+ \epsilon h^k Y+Z;\;\epsilon\in\{0,1\}, k\in \{1\ldots \alpha\}, Z\in Q\}
\]
for $h$ some power of 5, $X\in G^2$ and $Q$ some subgroup of $G^2$ containing all the elements
\begin{align*}
\left(s_{\alpha_1}(h)-h^{\alpha_1} s_{\alpha_2}(h)+\ldots+(-1)^{e+1} h^{\alpha_1+\ldots+\alpha_{e-1}}s_{\alpha_e}(h)\right)X  \\
+\left(-1+h^{\alpha_1}-h^{\alpha_1+\alpha_2}+ \ldots + (-1)^{e+1} h^{\alpha_1+\ldots+\alpha_{e}}\right)Y
\end{align*}
where the $0 <\alpha_i<|(h,X)|$ verify $\alpha_1+\ldots +\alpha_e=|(h,X)|$.
\end{example}
\begin{proof}
The subgroups of $\Aut(G)\cong \ZZ_{2}\times \ZZ_{2^{r-2}}$ are either cyclic and the computation of Example \ref{neq2} holds, or it is 
$<(1,0),(0,k)> \subset \ZZ_{2}\times \ZZ_{2^{r-2}}$ for some $k$ in $\ZZ_{2^{r-2}}$.
The corresponding subgroup of $\ZZ_{2^r}^\ast$ is $M=<-1,5^k>$. Define $h$ to be $5^k$.

Then straightforward computations show that an element $(1,Z)$ is in $<(h,X),(-1,Y)>$ if and only if it is in the subgroup of $\ZZ_p^2$ generated by all
\begin{align*}
\left(s_{\alpha_1}(m)-m^{\alpha_1} s_{\alpha_2}(m)+\ldots+(-1)^{e+1} m^{\alpha_1+\ldots+\alpha_{e-1}}s_{\alpha_e}(m)\right)X  \\
+\left(-1+m^{\alpha_1}-m^{\alpha_1+\alpha_2}+ \ldots + (-1)^{e+1} m^{\alpha_1+\ldots+\alpha_{e}}\right)Y
\end{align*}
where the $\alpha_i$ verify $\alpha_1+\ldots +\alpha_e=\alpha$.

The proof then follows by Proposition \ref{p3}.
\end{proof}

In the next section, we show that we can always assume $G$ is a $p$-group. As a result, we understand in detail the subgroup structure of $\ZZ_n$ for any integer $n$ (see Theorem \ref{zn}).

\pagebreak
\section{Subgroup structure of a product}
\label{product}
Let us state some general results, starting with the fundamental theorem of finite abelian groups (see for instance \cite{Lang}):
\begin{theorem}
\label{fund}
Let $G$ be a finite abelian group. Then $G$ is isomorphic to the product
of groups of the form
\[ H_p = \mathbb{Z}_{p^{r_1}} \times \ldots \times \mathbb{Z}_{p^{r_k}} \]
where $p$ is a prime integer and $r_1 \leq \ldots \leq r_k$.
\end{theorem}

\begin{lemma}
If $H$ and $K$ are finite groups with relatively prime orders, then
$\Aut(H\times K)\cong \Aut(H) \times \Aut(K)$.
\end{lemma}
\begin{proof}
see \cite{07}
\end{proof}
\begin{corollary}
\label{Cor1}
If $H$ and $K$ are finite groups with relatively prime orders, then
$\Atp(H\times K)\cong \Atp(H) \times \Atp(K)$.
\end{corollary}
\begin{proof}
\begin{align*}
\Atp(H\times K) &\cong \Aut(H\times K) \ltimes (H\times K)^2 \\
&\cong (\Aut(H) \times \Aut(K)) \ltimes (H^2 \times K^2) \\
&\cong (\Aut(H) \ltimes H^2)\times (\Aut(K) \ltimes K^2)\\
&\cong \Atp(H) \times \Atp(K).
\end{align*}
\end{proof}

Thus, $\Atp(G)$ is the direct product of all $\Atp(H_p)$. In particular, since $\Atp(H_p)\cong \Aut(H_p) \ltimes {H_p}^2$, there is a nice description of $\Atp(G)$, using the description of $\Aut(H_p)$ provided in \cite{07}.

\begin{lemma}
\label{Lem2}
If $H$ and $K$ are finite groups with relatively prime orders, then any subgroup of $H \times K$ is the product of a subgroup of $H$ with a subgroup of $K$.
\end{lemma}
\begin{proof}
This lemma is considered to be folklore, we prove it anyway (we do not pretend the proof is new).

Let $S$ be a subgroup of $H\times K$. Define $M_H \leq H$ and $M_K \leq K$ to be
\[
\begin{cases}
 M_H=\{ h\in H ;\; (h,1)\in S \} \\
 M_K=\{ k\in K ;\; (1,k)\in S \}
\end{cases}
\]

We have of course $M_H\times M_K\subset S$. We want to show that $|S|=|M_H| |M_K|$. Then the lemma follows clearly.

To prove this, notice that if for a fixed $k$ in $K$, $(h,k)$ and $(h',1)$ are in $S$, then
\[ (hh',k)=(h,k)(h',1)\in S\]

Moreover, if $(h,k)$ and $(h',k)$ are in $S$ then
\[ (hh'^{-1},1)=(h,k)(h',k)^{-1}\in S\]

Thus, 
\[ \{h\in H;\; (h,k)\in S\}\]
is always a coset of $M_H$. Therefore, 
\[ |S|=|M_H| |\{k\in K;\;(h,k)\in S \text{ for some }h\}| \]

Now, $\{k\in K;\;(h,k)\in S \text{ for some }h\}$ is a subgroup of $K$, so its order divides $|K|$.
Similarly, $|M_H|$ divides $|H|$ and $|S|$ divides $|H||K|$, say $|S|=mn$, $m$ dividing $|H|$, $n$ dividing $|K|$.
But then, $|H|$ and $|K|$ being relatively prime, we must have $m=|M_H|$.
Reasoning with $M_K$ instead of $M_H$, we see that we must also have $n=|M_K|$, so $|S|=|M_H| |M_K|$ and we are done.
\end{proof}

Combining Corollary \ref{Cor1} and Lemma \ref{Lem2}, we get
\begin{theorem}
\label{zn}
Let $G= \mathbb{Z}_{n}$, $n>0$. Let $n=p_1^{r_1} \ldots p_k^{r_k}$ be the prime decomposition of $n$. Define $G_i= \mathbb{Z}_{p_i^{r_i}}$ for all $i$. Then 
\begin{align*}
\Atp(G) & \cong \Atp(\mathbb{Z}_{p_1^{r_1}})\times \ldots \times \Atp(\mathbb{Z}_{p_k^{r_k}}) \\
    & \cong (G_1^\ast \ltimes G_1^2)\times \ldots \times (G_k^\ast \ltimes G_k^2).
\end{align*}
With this identification, any subgroup $S$ of $\Atp(G)$ is a product $S=S_1\times \ldots\times S_k$ of subgroups $S_i$ of $\Atp(G_i)$. Moreover, $S_i$ is either of the form
\[
S_i= \{ (h^k,s_{k}(h)X+Z);\; k \in \{1 \ldots |h|\},\, Z \in Q \}
\]
where $h\in G_i^\ast$, $X\in G_i^2$ and $Q$ is some subgroup of $G_i^2$ containing $s_{|h|}(h)X$, or, this latter case only happening when $G_i=\mathbb{Z}_{2^r}$, $r\geq3$,
\[
S_i=\{ (-1)^\epsilon h^k,s_k(h)X+ \epsilon h^k Y+Z;\;\epsilon\in\{0,1\}, k\in \{1\ldots \alpha\}, Z\in Q\}
\]
for $h$ some power of 5, $X\in G_i^2$ and $Q$ some subgroup of $G_i^2$ containing all the elements
\begin{align*}
\left(s_{\alpha_1}(h)-h^{\alpha_1} s_{\alpha_2}(h)+\ldots+(-1)^{e+1} h^{\alpha_1+\ldots+\alpha_{e-1}}s_{\alpha_e}(h)\right)X  \\
+\left(-1+h^{\alpha_1}-h^{\alpha_1+\alpha_2}+ \ldots + (-1)^{e+1} h^{\alpha_1+\ldots+\alpha_{e}}\right)Y
\end{align*}
where the $0 <\alpha_i<|(h,X)|$ verify $\alpha_1+\ldots +\alpha_e=|(h,X)|$.
\end{theorem}
\begin{proof}
Combine Theorem \ref{fund}, Corollary \ref{Cor1} and Proposition \ref{proposition} to prove that 
\[
\Atp(G) \cong (G_1^\ast \ltimes G_1^2)\times \ldots \times (G_k^\ast \ltimes G_k^2).
\]
Then conclude using Lemma \ref{Lem2} and Examples \ref{neq2} and \ref{eq2}. 
\end{proof}


\pagebreak
\section{Possible applications}
\label{generalization}
As highlighted in the introduction, every loop $L$ isotopic to a group $G$ is in fact isomorphic to it (see \cite{Pflu}).
Yet, understanding the subgroup structure of $\Atp(\ZZ_p)$, for $p$ prime, is a key ingredient in counting the number of latin squares of some type, as we wish to explain here.

The definition of quasigroups was provided in Section 2.2. Alternatively, they can be defined as groupoids $(G,\cdot)$ such that for any $a$, $b$ there are unique elements $c$, $d$ satisfying
\[
a \cdot c = b
\]
\[
d \cdot a = b
\]

We see immediately that Cayley tables of finite quasigroups are exactly latin squares, and Cayley tables of finite loops are exactly reduced latin squares. Sorting and enumerating quasigroups (or, equivalently, latin squares) is a major issue in latin square theory. One can choose to sort these up to isomorphy, but isotopy is often more suited. For instance, the fact that any quasigroup is isotopic to a loop (see \cite{Pflu}) can be stated in the form that any latin square can be put into a reduced form.

We reproduce Table 2 from \cite{Pflu}, comparing the number of classes of loops of order 1 to 7 up to isomorphy and isotopy. 
\begin{table}[h]
\begin{center}
\begin{tabular}{|l|l|l|l|l|l|l|l|}
\hline
order of loop				&1	&2	&3	&4	&5	&6	&7		\\ \hline
number of isomorphy classes	&1	&1	&1	&2	&6	&109&23,750	\\ \hline
number of isotopy classes	&1	&1	&1	&2	&2	&22	&563	\\ \hline
\end{tabular}
\end{center}
\caption{Number of classes of loops of order 1 to 7 up to isomorphy and isotopy.}
\end{table}

Now that we have recalled these facts about the importance of isotopy in latin square theory, applications of the present paper are easy to imagine: if we can enumerate finite quasigroups (resp. loops) of some kind, up to isomorphy, by making an essential use of the group of automorphisms of some abelian groups, then we can hope to mimic the proof leading to such an enumeration, and enumerate these quasigroups (resp. loops) up to isotopy instead. We are specifically interested in Theorem 7.3. of \cite{Petr} computing the number of nilpotent loops of order $2p$, $p$ prime. We claim that most of it can be translated for isotopy, and that Example \ref{zp} can be used in an essential way to find the number of nilpotent loops of order $2p$ \textit{up to isotopy}, instead of isomorphy.

\bibliography{biblio}		
\bibliographystyle{alpha} 
\end{document}